\documentclass[oneside,reqno, twoside, 11pt]{amsart}
\usepackage[pdftex]{graphicx}
\usepackage{color}
\usepackage[T1]{fontenc}
\usepackage{lmodern}
\usepackage[english]{babel}
\usepackage{mathrsfs}  

\usepackage{upgreek}

\usepackage{amsfonts}
\usepackage{extpfeil}
\usepackage{verbatim}

       \usepackage{BOONDOX-cal}

\usepackage{microtype}
\usepackage[new]{old-arrows}
\usepackage{amsmath}
\usepackage[all, cmtip]{xy}
\usepackage{tikz-cd}

\usepackage{amsthm}
\usepackage{mathtools}

\usepackage{bbm}
\usepackage{paralist}
\usepackage[T1]{fontenc}
\usepackage[colorlinks,citecolor=magenta,pagebackref=true,urlcolor=magenta, bookmarksdepth=3]{hyperref}
\usepackage[nodisplayskipstretch]{setspace}
\usepackage{breqn}

\usepackage{nicefrac}

\usepackage[font=small,labelfont=bf]{caption}
\usepackage[labelformat=simple]{subcaption}

\usepackage{epsfig}

\DeclareFontFamily{OT1}{pzc}{}
\DeclareFontShape{OT1}{pzc}{m}{it}{<-> s * [1.10] pzcmi7t}{}
\DeclareMathAlphabet{\mathpzc}{OT1}{pzc}{m}{it}

\makeatletter
\newtheorem*{rep@theorem}{\rep@title}
\newcommand{\newreptheorem}[2]{%
	\newenvironment{rep#1}[1]{%
		\def\rep@title{#2~\ref{##1}}%
		\begin{rep@theorem}}%
		{\end{rep@theorem}}}

\makeatother

\theoremstyle{plain}
\newtheorem*{thm*}{Theorem}
\newtheorem{thm}{Theorem}[section]
\newtheorem{mthm}{\bf Theorem}

\newtheorem{mcor}[mthm]{\bf Corollary}

\newreptheorem{mthm}{Main Theorem}
\newreptheorem{mcor}{Corollary}
\newreptheorem{mlem}{Lemma}
\newreptheorem{thm}{Theorem}

\newtheorem{cor}[thm]{Corollary}
\newtheorem{lem}[thm]{Lemma}

\newtheorem*{lem*}{Lemma}
\newtheorem{prp}[thm]{Proposition}
\newtheorem{conj}[thm]{Conjecture}

\newtheorem{prb}[thm]{Open Problem}

\newtheorem{obs}[thm]{Observation}

\theoremstyle{definition}
\newtheorem{dfn}[thm]{Definition}

\newtheorem{rem}[thm]{Remark}

\newcommand{\longtwoheadrightarrow}{\longrightarrow\hspace{-1.2em}\rightarrow\hspace{.2em}}

\newcommand{\ann}{\mathrm{ann}}

 \newcommand{\R}{\mathbb{R}}

\newcommand{\tet}{\vartheta}

\newcommand{\st}{\mr{st}}

\newcommand{\sbseq}{\subseteq}
\newcommand{\spseq}{\supseteq}

\newcommand{\n}{\frk{n}}

\newcommand{\vanish}[1]{}

\newcommand{\De}{\varDelta} 
\newcommand{\Ga}{\varGamma} 
 
\newcommand{\Sg}{\varSigma}

\def\V{{\bf V}}

\def\im{\mathrm{im}}
\def\ker{\mathrm{ker}}

\def\sbs\subset
\def\sbseq{\subseteq}

\def\langle{\left<}
\def\rangle{\right>}

\def\({\left(}
\def\){\right)}
\def\no={\,{\,|\!\!\!\!\!=\,\,}}

\def\no={\,{\,|\!\!\!\!\!=\,\,}}
\def\sbseq{\subseteq}

\def\sbseq{\subseteq}
\def\sbs\subset
\def\spseq{\supseteq}

\newcommand{\xqedhere}[2]{%
	\rlap{\hbox to#1{\hfil\llap{\ensuremath{#2}}}}}

\newcommand\Defn[1]{\textbf{#1}}
\newcommand{\cm}[1]{}

\newcommand\mbf[1]{\mathbf{#1}}

\newcommand\mr[1]{\mathrm{#1}}

\newcommand{\fld}{\mathbbm{k}}

\renewcommand\deg{\mr{deg}}
\newcommand{\bigslant}[2]{{\raisebox{.3em}{$#1$} \Big/ \raisebox{-.3em}{$#2$}}}

\newcommand\x{\mathbf{x}}

\DeclareMathOperator{\lk}{lk}
\DeclareMathOperator{\susp}{susp}

\newcommand{\KK}{\mathcal{K}}

\title[Anisotropy, biased pairings and Lefschetz for cycles]{Anisotropy, biased pairings, and the Lefschetz property for pseudomanifolds and cycles}
\author[Karim Adiprasito]{Karim Alexander Adiprasito}
\address{{Karim Adiprasito}, Einstein Institute of Mathematics, Hebrew University of Jerusalem, 91904 Jerusalem, Israel \emph{and} Department of Mathematical Sciences, University of Copenhagen, 2100 Copenhagen, Denmark}
\email{adiprasito@math.huji.ac.il}

\author[Stavros~A.~Papadakis]{Stavros~Argyrios~Papadakis}
\address{Stavros~Argyrios~Papadakis, Department of Mathematics, 
University of Ioannina,
Ioannina, 45110, 
Greece}
\email{spapadak@uoi.gr}

\author[Vasiliki Petrotou]{Vasiliki Petrotou}
\address{Vasiliki Petrotou,\ Department of Mathematics, 
University of Ioannina,
Ioannina, 45110, 
Greece}
\email{v.petrotou@uoi.gr}

\date{\today}


\begin{document}
	
	\begin{abstract}
We prove the hard Lefschetz property for pseudomanifolds and cycles in any characteristic with respect to an appropriate Artinian reduction. The proof is a combination of Adiprasito's biased pairing theory and a generalization of a formula of Papadakis-Petrotou to arbitrary characteristic. In particular, we prove the Lefschetz theorem for doubly Cohen Macaulay complexes, solving a generalization of the g-conjecture due to Stanley.
We also provide a simplified presentation of the characteristic 2 case, and generalize it to pseudomanifolds and cycles.
\end{abstract}
	
	\maketitle
	
	\newcommand{\AR}{\mathcal{A}}
	\newcommand{\BR}{\mathcal{B}}
	\newcommand{\CR}{\mathcal{C}}
	\newcommand{\Mu}{M}
	\newcommand{\Soc}{\mathcal{S}\hspace{-1mm}\mathcal{o}\hspace{-1mm}\mathcal{c}}
	\newcommand{\Socl}{{\Soc^\circ}}
	\renewcommand{\n}{\mbf{n}}

\section{Introduction}
	
In \cite{AHL}, Adiprasito introduced the Hall-Laman	relations and biased pairing property to prove the hard Lefschetz theorem for homology spheres and homology manifolds with respect to generic Artinian reductions and generic degree one elements acting as Lefschetz operators. They replace, intuitively speaking, the Hodge-Riemann relations in the K\"ahler setting, which we know cannot apply to general spheres outside the realm of polytope boundaries.

The proof relies on the fact that the biased pairing properties, the property that the Poincar\'e pairing does not degenerate at \emph{certain ideals}, specifically monomial ideals. This property allows for replacing the problem of proving the Lefschetz property with an equivalent one, which can be chosen much easier, as well as the idea of using the linear system as a variable. However, this replacement process is rather tedious geometrically, and is somewhat tricky to generalize to singular manifolds, and gets more tedious for objects beyond mild singularities. So the proof is geometrically somewhat challenging.

Recently, Papadakis and Petrotou \cite{PP} have built on the same two ideas, and studied anisotropy of Stanley-Reisner rings, a stronger (in the sense of more restrictive) notion than the biased pairing property, demanding that the Poincar\'e pairing does not degenerate at \emph{any} ideal. Clearly, for this, the field has to be rather special, and cannot be algebraically closed, for instance. Nevertheless, it implies a partial new proof of the Lefschetz property for spheres, though only in characteristic 2, by observing that in certain transcendental extensions, anisotropy is relatively easy  to prove by linear algebra arguments using a formula they had found.

Still, it is quite remarkable that both proofs of the generic Lefschetz property have their key ideas in common:

\begin{compactitem}[$\circ$]
\item Both proofs make use of the self-pairing in face rings. 
\item Both make use of infinitesimal deformations of the linear system.
\end{compactitem}

So it is natural to try to bring both together to compensate for each other's weaknesses. Our first main result is as follows.

\begin{mthm}\label{mthm:gl}
Consider $\fld$ any infinite field, $\mu$ a simplicial cycle of dimension $d-1$ over $\fld$, and the associated graded commutative face ring~$\fld[|\mu|]$ over its support $|\mu|$. 

Then there exists an Artinian reduction $\AR(\mu)$, and an element $\ell$ in $\AR^1(\mu)$, such that 
for every $k\le\nicefrac{d}{2}$, we have the \emph{hard Lefschetz property:} We have an isomorphism 
			\[\BR^k(\mu)\ \xrightarrow{\ \cdot \ell^{d-2k} \ }\ \BR^{d-k}(\mu). \]
\end{mthm}

Here, $\BR$ denotes the Gorensteinification of $\AR$, that is, the quotient of $\AR$ by the annihilator of the fundamental class. This generalizes the generic Lefschetz theorem for spheres (in which case $\BR^\ast(\Sg)= \AR^\ast(\Sg)$) and manifolds in \cite{AHL} and the characteristic two case for spheres in \cite{PP}. An important special case concerns pseudomanifolds. An \Defn{orientable pseudomanifold} is a pseudomanifold with a nontrivial fundamental class, with respect to a fixed characteristic. It is connected if that fundamental class is furthermore unique. 

\begin{mcor}\label{mthm:glp}
Consider $\fld$ any infinite field, $\Sg$ any $(d-1)$-dimensional orientable and connected pseudomanifold over $\fld$, and the associated graded commutative face ring~$\fld[\Sg]$. 

Then there exists an Artinian reduction $\AR(\Sg)$, and an element $\ell$ in $\AR^1(\Sg)$, such that 
for every $k\le\nicefrac{d}{2}$, we have the \emph{hard Lefschetz property:} We have an isomorphism 
			\[\BR^k(\Sg)\ \xrightarrow{\ \cdot \ell^{d-2k} \ }\ \BR^{d-k}(\Sg). \]
\end{mcor}

The Lefschetz property has many applications, implying the Gr\"unbaum-Kalai-Sarkaria conjecture, g-conjecture and many more, but we shall not discuss these here, and direct the interested reader to \cite{AHL} for a derivation of these implications. We will provide some new applications, to doubly Cohen-Macaulay complexes. Note that orientability is automatic over characteristic two. Alas, more miracles happen in that characteristic. We have the following stronger result:

	\begin{mthm}\label{mthm:as}
Consider $\fld$ any field of characteristic two, $\mu$ any $(d-1)$-dimensional cycle over $\fld$, and the associated graded commutative face ring~$\fld[|\mu|]$. Then, for some field extension $\fld'$ of $\fld$, we have an Artinian reduction $\AR(\mu)$ that is anisotropic, i.e. for every nonzero element $u\in \BR^k(\mu)$, $k\le \frac{d}{2}$, we have 
\[u^2\ \neq\ 0.\]
	\end{mthm}

That this is stronger than the Lefschetz theorem is a consequence of the characterization theorem of biased pairing theory, or alternatively the Kronecker/perturbation lemma, both of which we shall recall. In general characteristic, we are unable to prove such a statement. Instead, we prove that every $u$ pairs with another that is sufficiently similar to $u$, essentially related by a change in coefficients, so that they lie in the same monomial ideal.

\section{Basic Notions}\label{sec:basics}
	
Face rings are  the main object of the paper. Our treatment is standard except for the relative case, in which we follow \cite{AY}. Fix a field $\fld$. 

\begin{dfn}
	Let $\varDelta$  be a simplicial complex of dimension $d-1$. 
 Define the polynomial ring $\fld [x_v \mid v\in \varDelta^{(0)}]$, with variables indexed by vertices of $\varDelta$. The \Defn{non-face ideal} $I_\varDelta$ of $\varDelta$ is the ideal generated by all elements of the form $x_{v_1}\cdot x_{v_2} \cdot\ldots\cdot x_{v_j}$ where $\{v_1,\ldots,v_j\}$ is not a face of $\varDelta$. 	
	The \Defn{face ring} of $\varDelta$ is
	\[\fld [\varDelta] \coloneqq \fld [x_v \mid v\in \varDelta^{(0)}] / I_\varDelta.\]
	
	If $\Psi=(\varDelta,\Gamma)$ is a relative complex, the \Defn{relative face module} of $\Psi$ is defined by $I_\Gamma / I_\varDelta$. This is an ideal of $\fld [\varDelta]$.
\end{dfn}

To further out notational abilities, let us recall the definition of a star in a (relative) simplicial complex $\Psi=(\De,\Ga)$:
The star of a simplex $\tau$ within $\Psi$ is \[\mathrm{st}_\tau \Psi = (\{\sigma \in \De: \sigma \cup \tau \in \De\},\{\sigma \in \Ga: \sigma \cup \tau \in \Ga\}).\] Similarly, the link is 
\[\mathrm{lk}_\tau \Psi =  (\{\sigma \in \De: \sigma \sqcup \tau \in \De\},\{\sigma \in \Ga: \sigma \sqcup \tau \in \Ga\}).\]

Now, assume that $\fld$ is infinite. Consider an Artinian reduction $\AR^\ast(\varDelta)$ of a face ring $\fld[\varDelta]$ with respect to a linear system of parameters $\Theta$. It is instructive to think of $\AR^\ast(\varDelta)$ as a geometric realization of $\varDelta$ in $\fld^d$, with the coefficients of $x_i$ in $\Theta$ giving the coordinates of the vertex $i$, recorded in a matrix $\V$.

By definition, a simplicial cycle $\mu$ of $\Delta$ over $\fld$ is a nonzero element of  $H_{d-1}(\varDelta;\fld)$. 
Evaluating at $\mu$, we get a surjective $\fld$-linear map 
\[
   H^{d-1}(\varDelta;\fld) \twoheadrightarrow \fld.
\]
Via the canonical isomorphism
\[ 
          H^{d-1}(\varDelta;\fld)\ \cong\ \AR^d(\varDelta),
\]
see \cite[Section 3.9]{AHL}, \cite{TW},  there is an induced surjective $\fld$-linear map 
\[
    \mu^\vee :   \AR^d(\varDelta) \twoheadrightarrow \fld.
\]

Composing the multiplication map    $\; \AR^k(\varDelta)\ \times\ \AR^{d-k}(\varDelta) \to \AR^{d}(\varDelta) \;$  
with  $ \mu^\vee $,  we get a bilinear pairing 
\[
          \AR^k(\varDelta)\ \times\ \AR^{d-k}(\varDelta)\  \to \fld.
\]

We denote by  $L^k$ the  set of elements of    $\AR^k(\varDelta)$ that have zero pairing with all elements
of   $\; \AR^{d-k}(\varDelta)$.   The direct sum $\; L^\ast = \oplus_k L^k \;$ is a homogeneous 
ideal of $\; \AR^\ast(\varDelta)$,   and we set
\[
          \BR^\ast_\mu(\varDelta) =\AR^\ast(\varDelta) / L^\ast.
\]
The algebra $\AR^\ast(\varDelta)$ has been \Defn{Gorensteinified}, for lack of a better word.
This does not depend on the simplicial complex $\varDelta$. The following is immediate.
\begin{prp}
Consider a simplicial complex $\varDelta$ as above, and a cycle $\mu$ in it. Then the restriction
\[\AR^\ast(\varDelta)\ \longtwoheadrightarrow\ \AR^\ast(|\mu|)\]
where $|\mu|$ denotes the support of $\mu$, that is, the minimal simplicial complex containing $\mu$, induces an isomorphism of Gorensteinifications. In particular, we have
\[\BR^\ast_\mu(\varDelta)\ \cong\ \BR^\ast_\mu(|\mu|).\]
\end{prp}
For ease of notation, we shall abbreviate $\BR^\ast(\mu):= \BR^\ast_\mu(|\mu|)$.

For simplicial complexes such that   $\;\dim_{\fld}  H_{d-1}(\varDelta;\fld) = 1$, 
we set $\mu$ to be a fixed  nonzero  element of  $H_{d-1}(\varDelta;\fld)$.
This  is, in particular, the case for connected  orientable 
pseudomanifolds.

Notice that $\BR^\ast(\mu)$ is a Poincar\'e duality algebra, with last nonzero graded component 
in degree $d$.

We will now use  the simplicial cycle $\mu$ 
to  define a  natural  $\fld$-linear isomorphism
\[ 
   \deg:\BR^d(\mu)\ \longrightarrow\  \fld.
\]

\textbf{Convention.} Note that the degree map is readily described by the coefficients of the simplicial \emph{cycle}: It is enough to define it on cardinality $d$ faces $F$, as face rings are generated by squarefree monomials \cite[Section~4.3]{Lee} \footnote{Although Lee proves this only in characteristic zero, the argument goes through in general. We shall use his ideas several times in this paper for general characteristic, provided they apply there. We note that all the ideas we use readily extend, and without any modification of the arguments.}. And for a cardinality $d$ face $F$, we have
\[\deg(\x_F)\ =\ \frac{\mu_F}{|\V_{|F}|},\]
where $\mu_F$ is the oriented coefficient of $\mu$ in $F$, and we fix an order on the vertices of $\mu$ and compute the sign with respect to the fundamental class, and the determinant $|\V_{|F}|$ of the minor $\V_{|F}$ of $\V$ corresponding to $F$.

For instance, if $\mu$ is the fundamental class of a pseudomanifold, we have canonically
\[\deg(\x_F)\ =\ \frac{\mr{sgn}(F)}{|\V_{|F}|}.\] 

\textbf{Two perspectives.} There are, of course, two perspectives that we shall make use of: We can consider 
$\AR^\ast(\varDelta)$ over $\fld$ as functions in $\V$, including the degree map in particular, or we consider 
$\AR^\ast(\varDelta)$ over $\fld(\V)$, the field of rational functions associated to the transcendental extension of $\fld$ by the entries of $\V$. It is useful to keep this dichotomy in mind.

\section{Application: Doubly Cohen-Macaulay complexes and the cycle filtration}\label{sec:cycles}

Let us sketch the main application of Theorem~\ref{mthm:gl}: we show the Lefschetz theorem for doubly Cohen-Macaulay complexes.

Consider an $m$-dimensional subspace $\Mu$ of $H_{d-1}(\varDelta)$, or dually, a map 
\[ \Mu^\vee: H^{d-1}(\varDelta)\ \longtwoheadrightarrow\ \fld^m.\]
Then we have the quotient
$\BR^\ast(\Mu)$ of $\AR^\ast(\varDelta)$ induced as before as the quotient by the annihilator under the pairing
\[\AR^k(\Sg)\ \times\ \AR^{d-k}(\Sg)\ \longrightarrow\ \fld^m.\]

\begin{cor}
$\BR^\ast(\Mu)$ has the top-heavy Lefschetz property over every sufficiently large field extension of $\fld$ with respect to a sufficiently general position of parameters, that is, there is an $\ell \in \BR^1(\Mu)$ so that
\[\BR^k(\Mu)\ \xrightarrow{\ \cdot \ell^{d-2k}\ }\ \BR^{d-k}(\Mu)\]
is injective.
\end{cor}

\begin{proof}
Consider a generating system $(\mu_i)_{i\in I}$ for $\Mu$,
Then we have an injection
\[\BR^\ast(\Mu)\ \longrightarrow\ \bigoplus_{i\in I} \BR^\ast(\mu_i),\]
and hence
\[\begin{tikzcd}
 \BR^k(\Mu) \arrow{r}{\ \cdot  \ell^{d-2k}\ } \arrow[hook]{d}{} & \BR^{d-k}(\Mu) \arrow[hook]{d}{} \\
\bigoplus_{i\in I}\BR^k(\mu_i) \arrow{r}{\ \cdot  \ell^{d-2k}\ } & \bigoplus_{i\in I}\BR^{d-k}(\mu_i)
\end{tikzcd}
\]
which implies injectivity on the top if it is present on the bottom.
\end{proof}

\subsection{Doubly Cohen-Macaulay complexes}

A simplicial complex is called \Defn{$s$-Cohen-Macaulay} if it is Cohen-Macaulay and after the removal of $s$ vertices, the complex is still Cohen-Macaulay of the same dimension.

For instance, a triangulated homology sphere is $2$-Cohen-Macaulay, also called \Defn{doubly Cohen-Macaulay} \cite[Chapter III.3]{Stanley96}. Stanley showed that doubly Cohen-Macaulay complexes are level, that is, for such a complex $\Sg$ of dimension $d-1$, the socle is concentrated in degree $d$. In other words, if $\Mu=\AR^d(\Sg)$, we have
\[\BR^\ast(\Mu)\ =\ \AR^\ast(\Sg).\]
From the last result, we conclude

\begin{cor}
Consider a doubly Cohen-Macaulay complex $\Sg$ of dimension $d-1$. Then
$\AR^\ast(\Sg)$ has the top-heavy Lefschetz property over any infinite field extension of $\fld$ with respect to a sufficiently general position of parameters, that is, there is an $\ell \in \AR_\Mu^1(\Sg)$ so that
\[\AR^k(\Sg)\ \xrightarrow{\ \cdot \ell^{d-2k}\ }\ \AR^{d-k}(\Sg)\]
is injective.
\end{cor}

In particular, the $g$-vector of a doubly Cohen-Macaulay complex is an $M$-vector \cite{Macaulay}, since
\[\AR^k(\Sg)\ \xrightarrow{\ \cdot \ }\ \AR^{k+1}(\Sg)\]
is injective for $k\le \frac{d}{2}$.

\section{Biased pairings and Hall Laman relations}

Let us recall the basics of biased pairing theory. More depth and breath is found in \cite[Section 5]{AHL}, but we repeat proofs where they are needed for our purposes.

\subsection{Biased Poincar\'e pairings} Recall: Let $\mu$ be a $d-1$-dimensional simplicial cycle, and $\BR^\ast(\mu)$ the associated Gorenstein ring. Then we have a pairing
	\[\BR^k(\mu)\ \times\ \BR^{d-k}(\mu)\ \longrightarrow\ \BR^{d}(\mu)\ \cong\ \R.\]
We say that $\mu$ has the \Defn{biased pairing property} in degree $k\le \frac{d}{2}$ with respect to some proper subcomplex $\varGamma$ of the support, 
if the pairing
	\begin{equation}\label{eq:ospd}
	\KK^{k}(\mu,\Ga)\ \times\ \KK^{d-k}(\mu,\Ga)\ \longrightarrow\ \KK^{d}(\mu,\Ga)
	\end{equation}
	is nondegenerate on the left. Here
		$\KK^{k}(\mu,\Ga)$ is the nonface ideal of $\Ga$ in $\BR^\ast(\mu)$.
	Notice:
\begin{prp}
For an ideal $\mathcal{I}$ in $\BR^\ast(\mu)$
the following are equivalent:
\begin{compactenum}[(1)]
\item The map
\[\mathcal{I}\ \longrightarrow\ \bigslant{\BR^\ast(\mu)}{\ann_{\BR^\ast(\mu)} \mathcal{I} }\] is an injection in degree $k$.
\item The map
\[\mathcal{I}\ \longrightarrow\ \bigslant{\BR^\ast(\mu)}{\ann_{\BR^\ast(\mu)} \mathcal{I} }\] is a surjection in degree $d-k$.
\item For every $x\in \mathcal{I}^k$, there exists a $y$ in $\mathcal{I}^{d-k}$ such that $x\cdot y \neq 0$.
\item $\mathcal{I}$ has the biased pairing property in degree $k$.
\end{compactenum}
\end{prp}	

We obtain immediately an instrumental way to prove biased Poincar\'e duality for monomial ideals. 

\begin{cor}\label{cor:map}
$\KK^\ast(\mu,\varDelta)$ has the biased pairing property  in degree $k$ if and only if
\[\KK^k(\mu,\Ga)\ \longrightarrow\ \BR^k(\mu)_{\mid \overline{\Ga}}\]
is injective, where $\BR^\ast(\mu)_{\mid \overline{\Ga}}$ is the Poinc\'are dual of $\KK^k(\mu,\Ga)$, that is, the quotient of $\BR^k(\mu)$ by elements that pair only trivially with elements of $\KK^k(\mu,\Ga)$.
\end{cor}

Let us note separately a different condition.

\begin{cor}\label{cor:map2}
This is the case if and only if 
\[\KK^{d-k}(\mu,\Ga)\ \longrightarrow\ \BR^{d-k}(\mu)_{\mid \overline{\Ga}}\]
is surjective.
\end{cor}

\textbf{Biased pairings as "rewriting an element in terms of an ideal":} 
We have several criterions for the biased pairing property: other than the pairing question itself, Corollary~\ref{cor:map} asks whether an element of $\KK^{k}(\mu,\Ga)$ survives rewriting it in its Poincar\'e dual.
Corollary \ref{cor:map2} instead
asks us to rewrite elements  $\BR^{d-k}(\mu)_{\mid \overline{\Ga}}$ by elements in the ideal $\KK^{d-k}(\mu,\Ga)$. 

\subsection{Invariance under subdivisions}

An important tool is the invariance of biased pairing under subdivisions (and their inverses). We only need this for odd-dimensional cycles, but it holds regardless of that restriction, and with respect to the middle pairing, i.e., a manifold of dimension $2k-1$, and regarding the pairing in degree $k$.

Recall that a map of simplicial complexes is simplicial if the image of every simplex lies in a simplex of the image.

A simplicial map $\varphi:\mu' \rightarrow \mu$ of simplicial cycles is a \Defn{combinatorial subdivision} (speak: $\mu'$ is a subdivision of $\mu$) if it is a facewise injective simplicial map of underlying complexes and maps the fundamental class to the fundamental class. We have the following result of Barnette:

\begin{prp}[{Barnette, \cite{Barnette}}]\label{prp:Bar}
Any simplicial $d$-cycle $\mu$ is a subdivision of the boundary $\Delta_{d}$ of the $(d+1)$-dimensional simplex.

Moreover, if $F$ is any facet of $\mu$, and $v$ one of its vertices, then we can assume the combinatorial subdivision maps $F$ to a facet of $\Delta_{d}$, and the star of $v$ to the star of some vertex.
\end{prp}

For geometric simplicial complexes (that is, with respect to an Artinian reduction), we require a \Defn{geometric subdivision} to map the linear span of a simplex to the linear span of the simplex containing it, i.e. if $\sigma'$ is an element of $\mu'$, and $\sigma$ is the minimal face of $\sigma$ containing $\varphi(\sigma')$ combinatorially, then in the geometric realization, we require that the linear spans are mapped to corresponding linear spans, i.e.
\[\mr{span}\ \varphi(\sigma')\ \subseteq\ \mr{span}\ \sigma.\]
In particular, we do not require the image of a face to lie within the combinatorial target geometrically, only that they span the same space.

\begin{lem}\label{lem:PLinv}
	A geometric subdivision $\upvarphi:\mu'\rightarrow \mu$ of simplicial cycles of dimension at least $(2k-1)$ that restricts to the identity on a common subcomplex $\Ga$ preserves biased pairings, that is, 
	$\KK^\ast(\mu,\Ga)$ has the biased pairing property (in degree $k$) if and only if $\KK^\ast(\mu',\Ga)$ does (in degree $k$).	
\end{lem}

\begin{proof}
The map induces a pullback 
\[\varphi^\ast: \AR^\ast(|\mu |)\ \longhookrightarrow\ \AR^\ast(|\mu' |),\]
that is compatible with the Poincar\'e pairing, so it induces a map 
\[\varphi^\ast: \BR^\ast(\mu)\ \longhookrightarrow\ \BR^\ast(\mu').\]
Let us denote by $G$ the orthogonal complement to the image in $\BR^k(\mu')$, the image of the Gysin, so that the decomposition
\[\BR^k(\mu)\ \oplus\ G\ =\ \BR^k(\mu')\]
is orthogonal. Notice then that $\varphi^\ast$ is the identity on $\AR^\ast(|\mu |)$, and its image in $\BR^\ast(\mu )$, so
\[\KK^k(\mu,\Ga)\ \oplus\ G\ =\ \KK^k(\mu',\Ga)\]
Hence, it induces an isomorphism on the orthogonal complements of $\KK^\ast(\mu,\Ga)$ resp.\ $\KK^\ast(\mu',\Ga)$.
\end{proof}

Hence, if we do not subdivide $\Ga$, then the biased pairing property is preserved. If one subdivides $\Ga$, then only one direction holds.
\begin{prp}\label{prp:PLinv}
	Consider a geometric subdivision $\upvarphi:\mu'\rightarrow \mu$ of simplicial cycles of dimension at least $(2k-1)$, and $\Ga$ a subcomplex that is subdivided to $\Ga'$.
	
Then  $\KK^\ast(\mu,\Ga)$ satisfies biased Poincar\'e duality (in degree $k$) if $\KK^\ast(\mu',\Ga')$ does (in degree $k$).	
\end{prp}

The other direction is no longer true, as is easy to see.

\begin{proof}
It is still true that \[\BR^k(\mu)\ \oplus\ G\ =\ \BR^k(\mu')\]
in an orthogonal splitting under the Poincar\'e pairing. However, this is no longer so nice when restricting to the ideal, as we can only conclude that
\[\KK^k(\mu,\Ga)\ \oplus\ G'\ =\ \KK^k(\mu',\Ga')\]
where $G'$ is some subspace of $G$. However, as the pullback map is compatible with the pairing, the one direction we claimed still holds.
\end{proof}

%
%
%
%

\subsection{Hall-Laman relations and the suspension trick}

Finally, biased Poincar\'e duality allows us to formulate a Lefschetz property at ideals. We say that $\BR^\ast(\mu)$, with socle degree $d$, satisfies the \Defn{Hall-Laman relations} in degree $k\le \frac{d}{2}$ and with respect to an ideal $\mathcal{I}^\ast\subset \BR^\ast(\mu)$ if there exists an $\ell$ in $\BR^1(\mu)$, such that the pairing
	\begin{equation}\label{eq:sl}
	\begin{array}{ccccc}
	\mathcal{I}^k& \times &\mathcal{I}^k & \longrightarrow &\ \mathcal{I}^d\cong \R \\
	a	&		& b& {{\xmapsto{\ \ \ \ }}} &\ \mr{deg}(ab\ell^{d-2k})
	\end{array}
	\end{equation}
	is nondegenerate. Note that the Hall-Laman relations coincide with the biased pairing property if $k=\frac{d}{2}$.
	
	If we want to prove the Hall-Laman relations for a pair $(\mu, \Ga)$, where $\Ga$ is a subcomplex of $| \mu |$, specifically the Hall-Laman relations for
$\KK^\ast(\mu, \Ga)$ or its annihilator, we proceed using the following trick. Consider the suspension $\susp \De$ of a simplicial complex $ \De$. Label the two vertices of the suspension $\mbf{n}$ and $\mbf{s}$ (for north and south). Let $\uppi$ denote the projection along $\mbf{n}$, and let $\tet$ denote the height over that projection, and let $A\ast B$ denote the free join of two simplicial complexes $A$ and $B$.

	\begin{lem}[{\cite[Lemma 7.5]{AHL}}]\label{lem:midred}
Considering $\susp\mu$ realized in $\R^{d+1}$, and $k< \frac{d}{2}$, the following two are equivalent:
		\begin{compactenum}[(1)]
					\item The Hall-Laman relations for \[\KK^{k+1}(\susp \mu,(\susp \Ga )\cup \mbf{s} \ast |\mu|)	\]
					with respect to $x_{\mbf{n}}$.									
			\item The Hall-Laman relations for 
 \[\KK^{k}(\uppi\hspace{0.3mm}\mu,\uppi\hspace{0.3mm} \Ga)\]			
with respect to $\tet$. 
		\end{compactenum}
	\end{lem}	

\begin{proof}
	Set $\tet=x_{\mbf{n}}-x_{\mbf{s}}$ in  $\BR^{\ast}(\susp |\mu|)$.
Consider then the diagram
\[\begin{tikzcd}[column sep=5em]
 \BR^{k}(\uppi \mu ) \arrow{r}{\ \ \ \ \cdot \tet^{d-2k}\ \ \ \  } \arrow{d}{\sim } & \BR^{d-k}(\uppi \mu ) \arrow{d}{\sim } \\
\BR^{k+1}(\susp \mu ,\mbf{s}\ast \vert \mu \vert ) \arrow{r}{\ \ \ \cdot x_{\mbf{n}}^{d-2k-1}\ \ \ } & \BR^{d-k}(\susp \mu)_{\mid \overline{\mbf{s}\ast \vert \mu \vert}}
\end{tikzcd}
\]
An isomorphism on the top is then equivalent to an isomorphism of the bottom map, and the same holds when restricting to ideals and their Poincar\'e duals.
\end{proof}

\subsection{Lefschetz elements via the perturbation lemma}

Let us note we can use another way to construct Lefschetz elements.
For this, let us remember a Kronecker lemma of \cite{AHL}, see also \cite{Ringel} and of course Kronecker for the original formulation \cite{Kronecker}:

\begin{lem}\label{lem:perturbation}
		Consider two linear maps 
		\[\upalpha, \upbeta: \mathcal{X}\ \longrightarrow\ \mathcal{Y}\]
		of two vector spaces $\mathcal{X}$ and $\mathcal{Y}$ over $\R$.
 Assume that $\upbeta$ has image transversal to the image of $\upalpha$, that is,
			\[\upbeta(\ker\upalpha)\ \cap\ \im\upalpha\ =\ {0}\ \subset\ \mathcal{Y} .\]
			Then a generic linear combination 
			$\upalpha\ ``{+}"\ \upbeta$ of $\upalpha$ and $\upbeta$
			has kernel 
			\[\ker (
			\upalpha\ ``{+}"\ \upbeta)\ = \ \ker \upalpha\ \cap\ \ker \upbeta.\]
\end{lem}

As observed in \cite[Section 6.6]{AHL}, this can be used to iteratively prove the existence of Lefschetz elements provided that $\BR(\Sg)$ satisfies the biased pairing property in the pullback to any vertex link. Let us consider for simplicity the case of $\mu$ a cycle of dimension $2k-1$.

The connection is to the classical Hall matching theorem, which constructs stable matchings in a discrete setting \cite{Hall}. This lemma is designed to do the same in the setting of linear maps. The idea is now to prove the following \Defn{transversal prime property}: for $W$ a set of vertices in $| \mu |$ if 
\[\mathrm{ker}\ ``{\sum_{v\in W}}"\ x_v\ =\ \bigcap_{v\in W} \mathrm{ker}\ x_v\]

Note: proving the transversal prime property for all vertices together is equivalent to the Lefschetz isomorphism
\[ \mathcal{X}\ =\ \BR^{k}(\mu) \ \xrightarrow{\ \cdot \ell\ }\ \mathcal{Y}\ =\  \BR^{k+1}(\mu)\] 
for $\ell$ the generic linear combination over all variables. This is because  \[\bigcap_{v \text{ vertex of } \varSigma} \mathrm{ker}\ x_v \ =\ 0\] because of Poincar\'e duality.

Note further, to see how the biased pairing property implies the transversal property by induction on the size of the set $W$, that when we try to apply the criterion by multiplying with a new variable $x_v$, adding a vertex $v$ to the set $W$, then we are really pulling back to a principal ideal ideal $\langle x_v\rangle$ in $\BR(\mu)$, and asked to prove that 
$x_v \mathrm{ker}\ ``{\sum_{v\in W}}"$ and $\mathrm{im}\ ``{\sum_{v\in W}}"\cap \langle x_v\rangle$ intersect only in $0$. 

Note finally that both spaces are orthogonal complements. This is the case if and only if the Poincar\'e pairing is perfect when restricted to either (or equivalently both) of them. 

\section{Some useful identities on residues and degrees}\label{sec:diff}

We now prove and recall some useful identities on the degree.

\subsection{The square of a monomial}

Now that we know that we can restrict to minimal cycles of odd dimension $2k-1=d-1$, and biased pairings in them, we can go a little further. The trick is to consider the degree of an element in $\BR^k(\varDelta)$ of a pseudomanifold of dimension $d-1$ as a function of $\V$, which we think of as independent variables. Let us start with a formula due to Lee that describes the coefficients of the fundamental class:

\begin{lem}[{\cite[Theorem~11]{Lee}}]\label{lee:formula}
We have
\[\deg (\mathbf{x}_\tau^2) (\V)\ =\ \sum_{F \text{ facet containing } \tau } \deg(\x_{F}) \left(\prod_{i\in \tau} [F-i] \right)\cdot\left(\prod_{i\in F\setminus \tau} [F-i]^{-1}\right) \]
where $[F-i]$ is the volume element of $F-i$.
\end{lem}

To compute the volume element $[F-i]$, we can fix a general position vector $v$ that is added as an element to the matrix $\V$ in the $i$-th column, and compute the determinant.

Now consider $\deg (\mathbf{x}_\tau)^2(\V)$, for $\tau$ a face of cardinality $k$. We now want to compute the partial differential with respect to a $(k-1)$-dimensional face $\sigma$ of $\varDelta$. For this, we pick a basis of $\fld(\V)^d$  by simply considering the vertices $\bar \sigma_1,\ \dots, \ \bar \sigma_k$ of $\sigma$ and the faces for $ \tau_1,\ \dots, \tau_k$ of $\tau$. Denote the basis by $\mathcal{B}_\sigma^\tau$, simply a labelling of $\V_{|\tau\cup \sigma}$. 
Let us denote, for functions in $\V$, the partial differential 
\[\partial_\sigma^\tau\ :=\ \mr{deg}(\x_\sigma\x_\tau)\partial_{\mathcal{B}_\sigma^\tau},\]
where $\partial_{\mathcal{B}_\sigma^\tau}$ is the partial differential of $\V_{\sigma_1}$ in the directions of $\V_{\tau_1}$ and $\V_{\sigma_1}$, the partial differential of $\V_{\sigma_2}$ in directions $\V_{\tau_2}$ and $\V_{\sigma_2}$ and so on. In other words, we are interested in the behaviour of a function $f(\V)$ when varying $\V_{\sigma_1}$ in the directions of $\V_{\tau_1}$ and $\V_{\sigma_1}$, varying $\V_{\sigma_2}$ in directions $\V_{\tau_2}$ and $\V_{\sigma_2}$ etc. We have:

\begin{lem}
Assume $\sigma$ and $\tau$ are disjoint, but lie in a common face of $\varDelta$. Then
\[\partial_\sigma^\tau \deg (\mathbf{x}_\tau^2 (\V))\ =\ (\deg(\mathbf{x}_\tau \mathbf{x}_\sigma))^2 \]
\end{lem}

The normalization factor is added to achieve linearity in characteristic $2$, but is not needed for the proof of anisotropy and Lefschetz property. Let us briefly note that this formula follows also quite simply from the global residue formula \cite{Gr, CCD}, appropriately generalized to the setting of pseudomanifolds and face rings. This is not hard, but as Lee already provided a formula in the general setting, we shall work with him instead. The following easy derivation is due to Geva Yashfe:

\begin{proof}
Recall the generalized Leibniz formula
\[\frac{\partial}{\partial x} \prod f_i^{e_i}\ =\ \left( \prod f_i^{e_i} \right) \left ( \sum e_i \frac{\frac{\partial}{\partial x} f_i}{f_i} \right).\]
Applying this to Lee's formula, and differentiating $\sigma_j$ in the direction of $\tau_j$, every summand on the right vanishes except for a summand
\[(-1)^{k-1}\frac{[F-\sigma_j]}{[F-\tau_j]}.\]
Differentiating in direction $\sigma_j$ just multiplies everything with $-1$.
\end{proof}

We need another, different version, that is similarly simple to prove:

\begin{lem}
Assume $\sigma$ and $\tau$ are any two faces, and consider a vertex $v$ not in $\st_{\tau \cup \sigma}\varDelta$.
Then
$\deg (\mathbf{x}_\tau \mathbf{x}_\sigma)(\V)\ =\ 0$ is independent of
of $v$.
\end{lem}

\begin{proof}
This follows directly from Lee's formula.
\end{proof}

\subsection{The formula for a homogeneous element}

\newcommand{\tfld}{\widetilde{\fld}}

Set now $\tfld\ :=\ \fld(\V,\V')$, where $\fld$ is any field, the field of rational functions with variables $\V$, as well as a copy $\V'$ of the vertex coordinates. Consider now an element $u$ of $\fld(\V)^k[\varDelta]$ that is the linear combination of squarefree monomials. We say a face $\sigma$ is \Defn{compatible} with $u$ if 

\begin{compactitem}[$\circ$]
\item $\st_\sigma \varDelta$ intersects the support $|u|$ of $u$ in a unique face, denoted by $\tau(u,\sigma)$. The coefficient of  $u$ at $\tau(u,\sigma)$ is $1$.
\item Consider a face $\tau'$ of $|u|$ that is not $\tau(u,\sigma)$. Then the star of $\tau'$ intersects $\sigma$ trivially, or in a face $\sigma_{\tau'}$ of $\sigma$. Moreover, the coefficient $u_{\tau'}$ of $u$ in $\tau'$ vanishes under differentiation in the direction of vertices of $\sigma\setminus\sigma_{\tau'}$. 
\end{compactitem} 

Consider an element $u'$ of $\fld(\V')^k[\varDelta]$. We now differentiate $\deg(u\cdot u')$, using the formula of the previous section.

\begin{lem} For $u$ compatible with respect to $\sigma$, we have
\begin{equation}\label{eq:formula}
\partial_\sigma^{\tau(u,\sigma)} \deg(u\cdot u')\ =\ \deg((\partial_\sigma^{\tau(u,\sigma)}u)\cdot u')\ +\  (\deg(x_{\tau(u,\sigma)}\cdot \x_\sigma))^2\ 
\end{equation}
\end{lem}

The idea will be to cleverly associate $u'$ to $u$. Consider for instance the case when $u'$ is obtained from $u$ by replacing every variable in $\V$ with the corresponding variable in $\V'$. If we now  substitute $\V'\rightarrow \V$ on both sides of this equation, then we obtain

\begin{cor}\label{cor:form}

\begin{equation}\label{eq:formula2}
\partial_\sigma^{\tau(u,\sigma)} \deg(u\cdot u')_{u=u'}\ -\  \deg((\partial_\sigma^{\tau(u,\sigma)}u)\cdot u)\ =\  (\deg(x_{\tau(u,\sigma)}\cdot \x_\sigma))^2\ 
\end{equation}
\end{cor}

\textbf{Key Observation:} Note that if $u$ lies in some monomial ideal, then so does $u'$ and $\partial_\sigma^{\tau(u,\sigma)}u$, as we only changed the coefficients of the monomials, and introduced no new monomial. This is rather marvellous, and informs us how we want to prove the biased pairing property.

\section{Proving anisotropy and the Lefschetz property}

To finish the proof of the Lefschetz theorem, over $\tfld$, we need to prove the biased pairing property in degree $k$ for a pair $(\mu, \varDelta)$, where $\mu$ is a cycle of odd dimension $d-1=2k-1$, and $\varDelta$ is a codimension 0 subcomplex.

\subsection{Characteristic $2$, with a bonus of anisotropy}

In characteristic $2$, one can prove a stronger statement than just biased pairing: We can prove that elements in $u \in \BR^k(\mu)$ over $\fld(\V)$, for $\mu$ a pseudomanifold of dimension $2k-1$, have $\deg(u^2)\neq 0$. It illustrates an important principle: normalization.

Consider $u \in \AR^k(\mu)$. Consider the pairing with $\x_\sigma$ for some cardinality $k$ face $\sigma$. 

\textbf{Normalization:} We may now assume that, $u$ is represented as $\sum \lambda_\tau \x_\tau$ so that only one $\tau$ of the sum lies in $\st_\sigma \mu$, and may further assume that this $\tau$ lies in $\lk_\sigma \mu$. This is because \[\BR^k (\st_\sigma \mu)\ \cong\ \BR^k (\lk_\sigma \mu)\ \cong\ (\x_\sigma)\BR^k ( \mu)\] 
is of dimension one.

Finally observe that
\begin{align*}
\partial_\sigma^\tau \deg(u^2)\ &=\ \partial_\sigma^\tau \deg(\sum (\lambda_\tau \x_\tau)^2)\\
&=\ \sum \lambda_\tau^2 \partial_\sigma^\tau \deg( \x_\tau^2)\\
&=\ \sum \lambda_\tau^2 \deg(\x_\tau \x_\sigma)^2\\
 \ &=\  \sum \deg(\sum \lambda_\tau \x_\tau \x_\sigma)^2\ =  \ \deg(\x_\sigma u)^2
\end{align*}
in characteristic $2$. We conclude:

\begin{prp}
We have
\[\partial_\sigma^\tau \deg(u^2)\ =\ \deg(\x_\sigma u)^2\]
\end{prp}

Note that this holds regardless of the normalization, and in particular also of $\tau$, by linearity of the differential in characteristic $2$. On a low-brow level, this is a consequence of the vanishing of the diagonal terms in the Hessian in characteristic $2$. This finishes the proof of the biased pairing property, as every $u$ must pair with some~$\x_\sigma$ by Poincar\'e duality, so every element must pair with itself. 

\subsection{General characteristic, middle isomorphism} 

Let us consider the cycle $\mu'$ of dimension $2k-2$ in $\tfld^{[2k-1]}$, where $[j]=\{1, \cdots, j\}$, and assume we want to prove the middle isomorphism of the Lefschetz property. Following Lemma~\ref{lem:midred}, we can equivalently prove the biased pairing property in $\mu=\susp \mu'$ in $\tfld^{[2k]}$ with respect to the subcomplex  $\Ga \mbf{s} \ast |\mu' |$ after we lifted $\mu'$ according to some height function into $\tfld^{[2k]}$.

This is less beautiful: we do not have a differential formula that is independent of the presentation of an element in the reduced face ring. Instead, we shall need to construct a special presentation, and argue that there is an element in the ideal that pairs with it. What is more, that element will depend on $\sigma$, where $\x_\sigma$ pairs with the element in question (which again exists by Poincar\'e duality).

\textbf{Idea:} There are several ways to achieve this (a previous version of this paper provided a different argument, for instance), but we will make use of a particular choice of idea here: McMullen's idea of tracing the Lefschetz property along subdivisions and inverses, an influential idea that was previously thought to be difficult to make work in this setting. However, the idea will come with a twist, in that, essentially, many incompatible subdivisions have to be taken care of simultaneously, and are nonlocal.

Also, instead of showing that subdivisions can be used to preserve the Lefschetz property directly, we use the detour over biased pairings. 

\textbf{Setup:} Following Lemma~\ref{lem:midred}, we are tasked to prove the biased pairing property for a pair 
$\KK^\ast(\mu,\Ga)$. Consider an element $u$ generated by squarefree monomials supported not in $\Ga$, i.e., it is an element of $\tfld[\mu,\Ga]$. Such an element may not be compatible. Hence, the idea is to find an equivalent representation of the class $[u]$ of $u$ in $\KK^k(\mu,\Ga)$, and use the subdivision lemma and normalizations to replace it by an equivalent, but compatible element.

Let us now consider a class $[u]$ in $\KK^k(\mu,\Ga)$, represented by an element $u$ in the ideal $\langle\x_\n\rangle\subset\tfld[|\mu|]$ as a sum
\[u\ =\ \sum \lambda_\tau \x_\tau,\]
where the sum is over faces $\tau$ with common vertex $v$.

We want to prove that $u$ pairs nontrivially with some other element $\KK^k(\mu,\Ga)$. If it does so with some element $\x_\sigma$, associated to a cardinality $k$ face $\sigma$ containing $\n$, then we are done. Notice further that we may assume that
the coefficients $\lambda_\tau$ are from $\fld(\V)$, as the variables $\V'$ are transcendental over that field.

On the other hand, it has to pair with \emph{some} face of $\mu$ by Poincar\'e duality. It has to be one of the faces $\sigma$ of $\varSigma'$, as all other faces annihilate $\KK^\ast(\mu,\Ga)$. 

\textbf{Three steps:} We now proceed to find an element to pair $u$ with. This procedure depends on $\sigma$.

%
\textbf{Step 1. Normalization:} First, we use the following observation:

\begin{obs}
The quotient $\KK^k(\st_\sigma \mu, \st_\sigma \Ga)\ =\ \KK^k(\mu,\Ga)/\ann_{\x_\sigma} $ of $\KK^k(\mu,\Ga)$ is one-dimensional.
\end{obs}
In particular, we can assume that the restriction of $u$ to $\st_\sigma \mu$ is supported in a single face $\tau$ of cardinality $k$. 

Note that we can also normalize $u$, so that the coefficient of $u$ in that face is $1$. 

\textbf{Step 2. Specialization and unsubdivision:} Now, we shall do something counterintuitive to our philosophy:  We shall specialize the coordinates to a very special position. This is counter to our credo, that we have to be as generic as can be, but allows us to do a few interesting things.

Why can we do this?  A rational function that does not vanish under some specialization of its variables does not vanish. 

We now specialize and require that all the vertices of $|\mu'|$ lie in $\tfld^{[2k-1]}$. Moreover, we can use Barnette's Proposition~\ref{prp:Bar} to assume that $|\mu'|$ is geometrically the subdivision of the boundary of a $(2k-1)$-simplex, and that $\sigma \cup \tau\setminus \n$ is a facet of that simplex.

\begin{rem}[Desingularize]
There is one point we have to take care here: The specialization might singularize $u$. In such case, we can follow the discussion that follows, and take care that everything works at poles by careful calculation. The other option is this: We append another variable, say $\delta$, to $\tfld$, passing to the field extension $\tfld(\delta)$. This additional variable is then used to desingularize $u$, by perturbing it into a generic direction, so that we now have a function $u(\delta)$ with $u(0)=u$.
\end{rem}

So now we are stuck with $u$ in a subdivision of $(\n\ast \Delta_{2k-2},\Delta_{2k-2})$. Consider its projection along the Gysin to an image under the pullback map, and call the preimage $\overline{u}$. We observed already that if $\overline{u}$ pairs with an element in the ideal, then so is $u$; in fact we can simply pair $u$ with the pullback of whatever form $\upsilon$ the form $\overline{u}$ pairs with, as 
\[\deg(\overline{u}, \upsilon)\ =\ \deg(u, \upsilon)\]

 \textbf{Step 3. Normalization again:} Consider now $\overline{u}$, and assume that it pairs trivially with all cardinality $k$ faces of $\overline{\mu}$ that intersect $\sigma$ and contain $v$. We claim that this can essentially only happen if $u$ is compatible, though with a caveat. Consider first a facet $\sigma'$ of $\sigma$. Note that since we are in the simplex, we can assume that there is a unique face $A$ outside of $\tau$ that is in the star of $\sigma'$ and where the form is nontrivial.
 
Let us consider the coefficient in that face. Without saying anything further, it is not restricted, but we may assume that
$u$ \emph{does not} pair nontrivially with $\x_{\sigma'\ast \n}$ (otherwise we are done, as $\x_{\sigma'\ast \n}$ lies in $\AR(\n\ast \Delta_{2k-2},\Delta_{2k-2})$). 
The pairing then only involves the newly constructed face $A$ and the face $\tau(u,\sigma)$. It follows by Lemma~\ref{lee:formula} that the coefficient in $A$ vanishes under differentiation in the direction of vertices of $\sigma$.
Now, we repeat the same process for all nontrivial faces ${\sigma''}$ of $\sigma$, using pairing with the degree $k$ element $\x_{\sigma''} \times x_\n^{k-\text{card}{\sigma''}}$. Hence, $\overline{u}$ can be assumed to be compatible
and by Corollary~\ref{cor:form} we have that one of the two terms,
\[ \deg(\overline{u}\cdot \overline{u}')\ \quad \text{or}\ \quad \deg((\partial_\sigma^{\tau( \overline{u},\sigma)} u)\cdot \overline{u}) \]
is nontrivial. Note that this extends to the desingularization, as the combination of both terms is a rational function that, as we set $\delta$ to $0$, is a nontrivial function (after all, equalling to the degree of $x_\tau x_\sigma$, squared). Which proves biased pairing for $\overline{u}$, and hence for the original $u$. \qed

\begin{rem} Let us sketch another way that to construct the biased pair for $u$. 

First, we can assume that $u$ is supported in a unique face of $\st_\tau - \tau$ (there may be more faces in the support outside of the star). We can then assume by renormalizing so that $u$ does not become singular in the specialization, by multiplying with a scalar that only depends only on indeterminates we do not differentiate after by $\partial_\sigma^{\tau}$. Issue that arrises: the coefficient may go to $0$ on $\tau$. 

However, we can choose $u''$ to be $u'+ tx_\tau$, $t$ sufficiently large so that $(u'+ tx_\tau) u $ has coefficient $1$ on $\x_\tau^2$. It again becomes clear that  $\deg(\overline{u}\cdot \overline{u}'')$ or $\deg((\partial_\sigma^{\tau( \overline{u},\sigma)} u)$ are nontrivial.
\end{rem}

\subsection{General characteristic, general isomorphism}

In the case of general isomorphisms, say, the isomorphism 
\[\BR(\mu)^k\ \xrightarrow{\ \cdot \ell\ }\ \BR(\mu)^{d-k}\]
we are, instead of considering the first suspension, and ultimately the cone over a simplex, we are instead tasked with considering the $d-2k-1$-fold suspension of a simplex, and the cone over it. The rest of the argument remains the same.

\section{Open problems}
Concerning problems surrounding the $g$-theorem, there is a stronger conjecture available: It has been conjectured that stronger, for $s$-Cohen-Macaulay complexes, we have
this injection for $k\le \frac{d+s}{2}-1$, but the approach offers no clue how to do it. In particular, the combinatorial conjecture is:

\begin{conj}
Consider an $s$-Cohen-Macaulay complex of dimension $d-1$. Then the vector
\[(h_0,\ h_1- h_0,\ \dots,\ h_{\lceil\frac{d+s}{2}\rceil} -h_{\lceil\frac{d+s}{2}\rceil -1})\]
is an $M$-vector.
\end{conj}

It is equally an open problem to extend the anisotropy to general characteristic. We conjecture this is true, but have no good idea for an approach.

\begin{conj}
Consider $\fld$ any field, $\Sg$ any $(d-1)$-dimensional pseudomanifold over $\fld$, and the associated graded commutative face ring~$\fld[\Sg]$. Then, for some field extension $\fld'$ of $\fld$, we have an Artinian reduction $\AR(\Sg)$ that is totally anisotropic, i.e. for every nonzero element $u\in \AR^k(\Sg)$, $k\le \frac{d}{2}$, we have 
\[u^2\ \neq\ 0.\]
\end{conj}

Of course, special cases over characteristic $0$, such as the anisotropy of integral homology spheres over $\fld=\mathbb{Q}$ can be established using standard mixed characteristic tricks and reduction to characteristic $2$, but more seems difficult. A generalization of the anisotropy to characteristic $p$ that should be more immediate is the following:

\begin{conj}
Consider $\fld$ any field of characteristic $p$, $\Sg$ any $(d-1)$-dimensional pseudomanifold over $\fld$, and the associated graded commutative face ring~$\fld[\Sg]$. Then, for some field extension $\fld'$ of $\fld$, we have an Artinian reduction $\AR(\Sg)$ that is totally $p$-anisotropic, i.e. for every 
nonzero element $u\in \AR^k(\Sg)$, $k\le \frac{d}{p}$, we have 
\[u^p\ \neq\ 0.\]
\end{conj}

Finally, both \cite{AHL} and the present work only prove the existence of Lefschetz elements on a \emph{generic} linear system of parameters, and it would be interesting to have specific ones. One candidate is, in our opinion, the moment curve $(t,t^2, \dots, t^d)$.

\begin{prb}
Do distinct points on the moment curve provide a linear system with the Lefschetz property (on any (PL)-sphere, pseudomanifold or cycle)?
\end{prb}

\textbf{Acknowledgements.} We thank Christos Athanasiadis, David Eisenbud, Gil Kalai, Eran Nevo, Richard Stanley, Geva Yashfe and Hailun Zheng for some enlightening conversations, as well as helpful comments on the first version. K. A.  was supported by the European Research Council under the European
Union's Seventh Framework Programme ERC Grant agreement ERC StG 716424 - CASe and the Israel Science Foundation under ISF Grant 1050/16. 
We benefited from
experiments with the computer algebra program Macaulay2 \cite{M2}. 
This work is part of the Univ. of Ioannina Ph.D. thesis of V. P., financially 
supported by the Special Account for Research
Funding (E.L.K.E.) of  the University of Ioannina under the program with code number
82561 and title \emph{Program of financial support for Ph.D. students and postdoctoral researchers}.

	{\small
		\bibliographystyle{myamsalpha}
		\bibliography{ref}}

\end{document}